\documentclass[reqno,a4paper]{amsart}

\usepackage{microtype}
\usepackage{amsmath,amssymb,amsthm,mathtools,xcolor,graphicx}

\numberwithin{equation}{section}

\newtheorem{theorem}{Theorem}[section]
\newtheorem{proposition}[theorem]{Proposition}
\newtheorem{lemma}[theorem]{Lemma}

\theoremstyle{remark}
\newtheorem{remark}[theorem]{Remark}

\newcommand{\R}{\mathbb{R}}

\newcommand{\dd}{\mathop{}\!{d}}
\newcommand{\ii}{i}
\newcommand{\ee}{e}
\newcommand{\A}{\mathbf{A}}

\newcommand{\spec}{\sigma}
\newcommand{\specess}{\sigma_{\mathrm{ess}}}

\DeclareMathOperator{\curl}{curl}
\DeclarePairedDelimiter{\abs}{\lvert}{\rvert}
\DeclarePairedDelimiter{\norm}{\lVert}{\rVert}
\DeclarePairedDelimiter{\set}{\lbrace}{\rbrace}

\title
  [Magnetic bound states in sectors]
  {Bound states for the\\magnetic Neumann Laplacian\\ in planar sectors}

\author[A. Kachmar]{Ayman Kachmar}
\address
  [A. Kachmar]
  {Department of Mathematics {\normalfont and} PDE Research Unit–Center for
  Advanced Mathematical Sciences (CAMS)\newline American University of Beirut,
  P.O. Box 11-0236 Riad El-Solh / Beirut 1107 2020 Lebanon.}
\email{ak292@aub.edu.lb}

\author[M. Sundqvist]{Mikael Sundqvist}
\address
  [M. Sundqvist]
  {Department of Mathematics, Lund University, Sweden}
\email{mikael.persson\_sundqvist@math.lth.se}

\subjclass[2020]{35P15, 35J10, 81Q10}
\keywords{Magnetic Neumann Laplacian, planar sector, discrete spectrum,
de Gennes constant, bound state}

\begin{document}

\begin{abstract}
We study the magnetic Neumann Laplacian in an infinite planar sector of
opening $\alpha\in(0,\pi)$ under a constant magnetic field. Building on
earlier work by Bonnaillie-No\"el and collaborators and by Exner,
Lotoreichik, and P\'erez-Obiol, we prove that the bottom of the spectrum
lies strictly below the half-plane threshold for every convex sector.
Consequently, $H_\alpha$ has a discrete ground-state eigenvalue for every
$0<\alpha<\pi$. This resolves the bound-state problem for convex sectors,
a model problem arising in the analysis of magnetic localization near
corners and of the third critical field in type-II superconductivity.
\end{abstract}

\maketitle

\section{Introduction and the main result}

For $0 < \alpha < \pi$, set
\[
  \Omega_\alpha
  =\set{(x,y)\in\R^2 \mid \abs {y} < x\tan (\alpha/2)}.
\]
Throughout the paper the magnetic field has strength one. Let
\[
  \A _ 0(x,y)=\frac12(-y,x),
  \qquad \curl\A _ 0=1,
\]
and let $H_\alpha$ be the self-adjoint operator associated with
\[
  q_\alpha[u]
  =
  \int_{\Omega_\alpha}
  \abs{(-\ii\nabla + \A _ 0)u}^2\,\dd x\,\dd y,
  \qquad
  u\in H^1_{\A _ 0}(\Omega_\alpha).
\]
Here
\[
  H^1_{\A_0}(\Omega_\alpha)
  =
  \set{
    u\in L^2(\Omega_\alpha)
    \mid
    (-\ii\nabla+\A_0)u\in L^2(\Omega_\alpha;\mathbb C^2)
  };
\]
the magnetic Neumann boundary condition is encoded by this form domain.

We write
\[
  \lambda(\alpha) = \inf\spec(H_{\alpha}).
\]
The essential-spectrum threshold is the half-plane threshold~\cite{MR2127997},
\[
  \inf\specess(H_{\alpha}) = \Theta_0,
  \quad 0 < \alpha < \pi,
\]
where $\Theta_0 \approx 0.59$ is a spectral constant defined in
Section~\ref{sec:de-gennes}.

\begin{theorem}\label{thm:main}
For every $0<\alpha<\pi$,
\[
  \lambda(\alpha) < \Theta_0.
\]
Consequently, $\lambda(\alpha)$ is a discrete eigenvalue of $H_{\alpha}$.
\end{theorem}

\begin{remark}
There is no loss of generality in fixing the magnetic field strength to be one.
If $H_{\alpha,b}$ denotes the corresponding operator for a constant field $b>0$
and
\[
  \lambda_b(\alpha)=\inf\spec(H_{\alpha,b}),
\]
then dilation invariance of the sector gives
\[
  H_{\alpha,b}\simeq bH_\alpha,
  \qquad
  \lambda_b(\alpha)=b\lambda(\alpha).
\]
In particular, the essential-spectrum threshold is $b\Theta_0$.
\end{remark}

\begin{remark}
The endpoint $\alpha=\pi$ cannot be included in Theorem~\ref{thm:main}. In that
case, $\Omega_\pi$ is a half-plane and
\[
  \inf\spec(H_\pi)=\Theta_0,
\]
but $\Theta_0$ is not an $L^2$-eigenvalue,~\cite{MR2662319}.
\end{remark}

The magnetic Neumann Laplacian on a sector is the local model for a corner in the
strong-field, or equivalently semiclassical, analysis of magnetic Schr\"odinger
operators. Its lowest energy determines whether magnetic states localize near the
corner rather than along a smooth portion of the boundary. This mechanism is
important in the Ginzburg--Landau theory of type-II superconductors: close to the
third critical field, the geometry of the boundary influences where
superconductivity persists, and corners can act as preferred nucleation
sites~\cite{MR2231969,MR2340675}, a scenario that is conditional on
$\lambda(\alpha)<\Theta_0$, as now proved in Theorem~\ref{thm:main} above.

The foundational analysis of the sector problem for general opening angles was
carried out by Bonnaillie~\cite{MR2127997}; for the special case $\alpha=\pi/2$,
see also \cite{MR1911825, MR1852538}. Numerical computations by
Bonnaillie-No\"el, Dauge, Martin, and Vial~\cite{MR2340008} subsequently gave
strong evidence that
\[
  \lambda(\alpha)<\Theta_0
  \qquad\text{for every }0<\alpha<\pi,
\]
and that $\lambda(\alpha)$ is an increasing function of $\alpha$. Exner,
Lotoreichik, and P\'erez-Obiol later established the existence of a bound state
for a broad range of opening angles by a variational argument and extended this
range further ($\alpha< 0.595\pi$) with computer assistance~\cite{MR3906610}.
More recently, Bonnaillie-No\"el, Fournais, Kachmar, and Raymond proved the
existence of bound states for almost-flat corners, corresponding to $\alpha$
sufficiently close to $\pi$ from below~\cite{MR4780715}.

\begin{figure}[htb]
  \includegraphics[page=2]{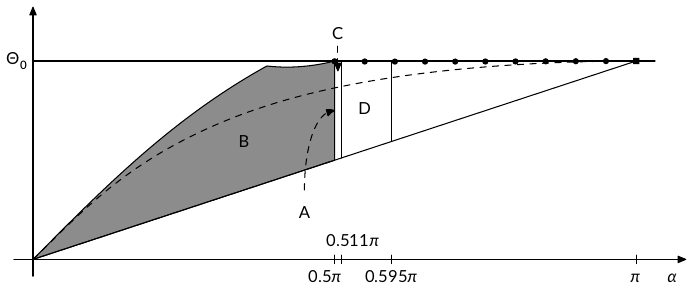}
  \caption {Various results for the sector.
  A: for $\alpha = \pi/2$~\cite {MR1852538}.
  B: for $0 < \alpha < \pi/2$~\cite {MR2127997}.
  C: for $\alpha$ up to $0.511\pi$~\cite {MR2017694}.
  D: for $\alpha$ up to $0.595\pi$~\cite {MR3906610}.
  The marked dots are values we get when calculating the energy of our
  trial state with $\epsilon = 0.01$. The small square at $\alpha = \pi$
  recalls the result from~\cite{MR4780715}.}
  \label {fig:sectorresults}
\end{figure}

Theorem~\ref{thm:main} completes this picture for convex sectors. Theorem~1.3
of~\cite{MR3906610} already covers, in particular, the interval
$0<\alpha\leq\pi/2$. It therefore remains to prove the strict inequality for
\[
  \frac{\pi}{2}\leq\alpha<\pi,
\]
which is the purpose of the variational construction below. The proof relies on
the construction of a trial state. However, unlike~\cite{MR4321415,MR4780715},
the setting of Theorem~\ref{thm:main} does not allow a truncated phase factor in
the trial state; this forces a global construction and a more refined energy
estimate.

The proof of Theorem~\ref{thm:main} is given in the next section. We then
conclude by collecting several applications of this result.

\section {Proof of Theorem~\ref {thm:main}}

Throughout the proof we assume that $\pi/2\leq\alpha<\pi$. We construct a trial
function with Rayleigh quotient strictly below $\Theta_0$. We begin by recalling
the properties of the de Gennes operator needed in the construction.

\subsection {The de Gennes operator and moment identities}\label{sec:de-gennes}

For $\xi\in\R$, consider the Neumann harmonic oscillator
\[
  \mathfrak {h}(\xi)
  =-\frac {\dd^2}{\dd t^2}+(t-\xi)^2
  \quad\text{in }L^2(\R_+),
\]
with Neumann boundary condition at $t = 0$. Let $\mu(\xi)$ be its lowest
eigenvalue, let $\xi _ 0 > 0$ be its unique minimizer, and define
\[
  \Theta _ 0=\mu(\xi _ 0).
\]

Let $\varphi_0>0$ be the corresponding normalized eigenfunction, and set
\[
  C_1=\frac{\varphi_0(0)^2}{3}.
\]
We also recall that $\Theta_0=\xi_0^2$.

Define the moments
\[
  M_n=\int_0^\infty (t-\xi_0)^n\varphi_0(t)^2\,\dd t.
\]
The standard identities needed below are collected in the following
lemma; see~\cite{MR2662319}.

\begin{lemma}\label{lem:moments}
The following identities hold:
\begin{equation}
  M_1 = 0,\quad
  M_2 = \frac {\Theta _ 0}{2},\quad
  M_3 = \frac {C_1}{2},\quad
  M_4 = \frac {3}{8} (1 + \Theta _ 0 ^ 2 - 3 \xi _ 0 C _ 1 ).
  \label{eq:moment-1234}
\end{equation}
Moreover, with
\[
  h(t) = \varphi _ 0'(t)^2 + (t - \xi _ 0) ^ 2 \varphi _ 0(t) ^ 2 - \Theta _ 0 \varphi _ 0(t)^2,
\]
one has
\begin{equation}
  \int _ 0 ^ {+\infty} h(t)\dd t = 0,
  \quad
  \int _ 0 ^ {+\infty} t h(t)\dd t = \frac{\varphi _ 0(0)^2}{2} = \frac{3C_1}{2}.
  \label{eq:weighted-energy}
\end{equation}

\end{lemma}

\begin{proof}
The moment identities are standard; see~\cite{MR2662319}. The first identity
in~\eqref{eq:weighted-energy} follows by integrating the eigenvalue equation.
Multiplying the de Gennes equation by $t\varphi _ 0$ and integrating by parts
yields the second identity.
\end{proof}

We recall the following rigorous bounds from~\cite{MR2912745}:
\[
  \frac{5901}{10000} < \Theta _ 0 < \frac{5902}{10000},
  \quad
  \frac{7681}{10000} < \xi _ 0 < \frac{7682}{10000},
  \quad
  \frac{253}{1000} < C _ 1 < \frac{255}{1000}.
\]
In fact, for us it will be sufficient to use the weaker bounds
\begin{equation}\label{eq:coarse-bounds}
  \Theta _ 0 < \frac {3}{5},
  \quad
  \xi _ 0 > \frac {3}{4},
  \quad
  C _ 1 > \frac {1}{4}.
\end{equation}

\subsection {Half-sector coordinates and phase matching}

We first reduce to one half of the sector. Set
\[
  \beta=\frac\alpha2,
  \quad k=\tan\beta.
\]
Then $k\geq1$.

\begin{figure}[htb]
  \includegraphics[page=1]{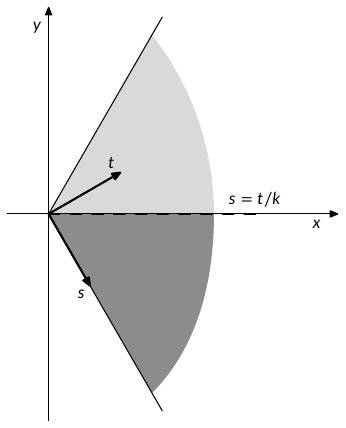}
  \caption {The sector, with the half we work in marked in darker gray. The
  coordinate systems are also shown, and the symmetry line $y = 0$ ($s = t/k$) is
  dashed.}
  \label {fig:coordinates}
\end{figure}

On the lower half of the sector introduce orthonormal coordinates
$(s,t)$ by
\begin{equation}\label{eq:coordinates}
  (x,y)
  =s(\cos\beta,-\sin\beta)
   +t(\sin\beta,\cos\beta).
\end{equation}
Here $s$ is the coordinate along the lower boundary ray and $t$ is the inward
normal coordinate, see Figure~\ref{fig:coordinates}. The lower half-sector
becomes
\[
  D _ k = \set{(s,t) \mid t > 0,\ s > t/k},
\]
and the line $s=t/k$ is the bisector $y = 0$. If $R_\beta$ denotes the rotation
in~\eqref{eq:coordinates}, then
\[
  R _ {\beta} ^ {T} \A_0(R_\beta(s,t))
  =
  \frac12(-t,s).
\]
Thus, in these coordinates the symmetric gauge is given by
\[
  \A_{\mathrm {sym}} = \frac {1}{2} (-t,s).
\]
We also use the Landau gauge
\[
  \A _ L = (-t,0),
  \quad
  \A _ {\mathrm {sym}}
  =\A _ L + \nabla\left(\frac{st}{2}\right).
\]
Thus a Landau-gauge function $v$ corresponds to the symmetric-gauge function
\begin{equation}\label{eq:gauge-transform}
  u = \ee ^ {-\ii st/2}v.
\end{equation}
Define
\[
  G(t)
  =
  \frac12\bigl((t-\xi _ 0)^2-\Theta _ 0\bigr)
  =
  \frac {t ^ 2}{2} - \xi _ 0 t,
\]
where the last equality follows from $\Theta _ 0=\xi _ 0 ^ 2$.  We shall choose a
real phase correction $\chi$ satisfying
\begin{equation}\label{eq:phase-boundary-condition}
  \chi\left(\frac tk,t\right)
  =
  \frac{G(t)}k.
\end{equation}
On the bisector, $s = t/k$, the total symmetric-gauge phase then vanishes (see
the trial state in~\eqref{eq:trial-v} below):
\begin{equation}\label{eq:real-trace}
  \xi _ 0 s+\chi(s,t)-\frac{st}{2}
  =\frac{\xi _ 0 t}{k}+\frac{t^2/2-\xi _ 0 t}{k}
   -\frac{t^2}{2k}
  =
  0.
\end{equation}
Consequently, any trial function with this phase has a real trace on
$y=0$. If $u^-$ denotes its restriction to the lower half-sector, define
its extension to the upper half-sector by
\begin{equation}\label{eq:reflection}
  u ^ {+} (x,y) = \overline {u ^ {-} (x,-y)}.
\end{equation}
The two traces agree by~\eqref{eq:real-trace}, hence the two pieces glue together
to define a function in the magnetic form domain. A direct calculation in the
symmetric gauge shows that reflection followed by complex conjugation preserves
both the $L ^ 2$ norm and the magnetic energy. The energy and norm of the
full-sector trial function are therefore twice those of its lower-half
restriction.

\subsection{The phase and a limiting corner defect}

Put $z = s - t/k \geq 0$ and define
\begin{equation}\label{eq:explicit-phase}
  \Psi(z,t) = \ee ^ {-2z}[G(t) + z(t - \xi _ 0)],
  \quad
  \chi(s,t) = \frac {1}{k} \Psi(s - t/k,t).
\end{equation}
At $z=0$, the phase satisfies $\chi(t/k,t) = G(t)/k$, as required
by~\eqref{eq:phase-boundary-condition}. Moreover, for each fixed $t$, $\chi(s,t)$
decays exponentially as $s\to+\infty$. For $\epsilon>0$, define in the Landau
gauge
\begin{equation}\label{eq:trial-v}
  v _ {\epsilon}(s,t)
  =
  \ee ^ {-\epsilon s}\varphi _ 0 (t) \ee^{\ii(\xi _ 0 s+\chi(s,t))},
  \quad (s,t)\in D_k.
\end{equation}
Applying the gauge transformation~\eqref{eq:gauge-transform} and then
the reflection rule~\eqref{eq:reflection} yields a full-sector trial
function in the magnetic form domain, which we denote by $u_\epsilon$.

Set
\[
  E_k[\chi]
  =
  \int _ {D_k} \varphi _ 0 (t) ^ 2
  \bigl[\chi _ s(s,t) ^ 2 + \chi _ t(s,t) ^ 2 \bigr] \dd s\dd t.
\]
We write
\[
  q^L_{D_k}[v]
  =
  \int_{D_k}
  \left(
    \abs{(-\ii\partial_s-t)v}^2
    +
    \abs{\partial_t v}^2
  \right)\dd s\,\dd t
\]
for the quadratic form in the Landau gauge $\A _ L=(-t,0)$.

\begin{proposition}\label{prop:defect}
For the lower-half trial function~\eqref{eq:trial-v},
\begin{equation}\label{eq:defect-formula}
  \lim_{\epsilon\to 0 ^ +}
  \bigl(q ^ L _ {D _ k}[v_\epsilon]
        -\Theta _ 0 \norm {v _ \epsilon } _ {L ^ 2(D _ k)} ^ 2\bigr)
  =
  -\frac {C _ 1}{k} + E _ k[\chi].
\end{equation}
\end{proposition}

\begin{proof}
Expanding the form gives
\begin{align*}
  q ^ L _ {D _ k}[v_{\epsilon}]
  - \Theta _ 0\norm {v _ {\epsilon}} ^ 2
  ={}& \epsilon ^ 2
      \int _ {D _ k} \ee ^ {-2\epsilon s}\varphi _ 0(t) ^ 2\dd s\dd t
  \\
  & + \int _ {D _ k} \ee ^ {-2\epsilon s}h(t) \dd s\dd t
  \\
  & + 2\int _ {D _ k} \ee ^ {-2\epsilon s}
       (\xi _ 0 - t)\varphi _ 0(t)^2 \chi _ s(s,t) \dd s\dd t
  \\
  & + \int _ {D _ k} \ee ^ {-2\epsilon s} \varphi _ 0 (t)^2
       [\chi _ s (s,t) ^ 2 + \chi _ t (s,t) ^ 2] \dd s\dd t.
\end{align*}
Note that there is no mixed term involving $\epsilon$; the amplitude
derivative and the phase-current term are respectively purely imaginary and real.

The first term tends to zero because
\[
  \epsilon ^ 2 \int _ {D _ k} \ee ^ {-2\epsilon s}\varphi _ 0(t) ^ 2\dd s\dd t
  =
  \frac {\epsilon}{2}
    \int _ 0 ^ {+\infty} \ee ^ {-2\epsilon t/k} \varphi _ 0 (t) ^ 2\dd t.
\]
For the second term, using~\eqref{eq:weighted-energy},
\[
  \int _ {D _ k} \ee ^ {-2\epsilon s}h(t)\dd s\dd t
  =\frac {1}{2\epsilon}
    \int _ 0 ^ {+\infty} \ee ^ {-2\epsilon t/k} h(t)\dd t
  \to
  -\frac{1}{k} \int _ 0 ^ {+\infty} t h(t)\dd t
  =
  -\frac{3C_1}{2k}.
\]
For the linear phase term, dominated convergence
and~\eqref{eq:phase-boundary-condition} yield
\begin{align*}
  \MoveEqLeft
  2\int_{D_k} \ee ^ {-2\epsilon s}
       (\xi _ 0 - t) \varphi _ 0 (t) ^ 2\chi _ s(s,t)\dd s\dd t
  \\
  &
  \to
  2 \int _ 0 ^ {+\infty} (\xi _ 0 - t) \varphi _ 0(t) ^ 2
     \biggl(\int _ {t/k} ^ {+\infty} \chi _ s (s,t)\dd s \biggr)\dd t
  \\
  &
  =
  - \frac {2}{k} \int _ 0 ^ {+\infty} (\xi _ 0 - t) G(t) \varphi _ 0(t) ^ 2\dd t.
\end{align*}
Since $G(t) = [(t - \xi _ 0) ^ 2 - \Theta _ 0]/2$, Lemma~\ref{lem:moments} gives
\[
  \int _ 0 ^ {+\infty} (\xi _ 0 - t) G(t) \varphi _ 0(t) ^ 2\dd t
  =
  -\frac12(M_3-\Theta _ 0 M_1)
  =
  -\frac{C_1}{4}.
\]
Thus the linear phase term converges to $C_1/(2k)$. Finally, the last term
converges to $E_k[\chi]$. Combining the four limits
proves~\eqref{eq:defect-formula}.
\end{proof}

It remains to verify that the phase $\chi$ defined in~\eqref{eq:explicit-phase}
satisfies $E_k[\chi]<C_1/k$.

\subsection{Verifying the remaining inequality}

Here $\chi_t$ denotes differentiation at fixed $s$, whereas $\Psi_t$
denotes differentiation at fixed $z$. By the chain rule,
\[
  \chi _ s
  =
  \frac {1}{k} \Psi _ z,
  \quad
  \chi _ t
  = \frac {1}{k}\bigl(\Psi _ t - \frac{1}{k} \Psi _ z\bigr).
\]
Therefore,
\[
  E_k[\chi] = \frac1{k^2}J(c), \quad c=\frac {1}{k} \in (0,1],
\]
where
\begin{equation}\label{eq:J-definition}
  J(c)
  =
  \int _ 0 ^ {+\infty}\int _ 0 ^ {+\infty} \varphi _ 0 (t) ^ 2
  [\Psi _ z ^ 2 + (\Psi _ t - c\Psi _ z) ^ 2] \dd z\dd t.
\end{equation}
From
\[
  \Psi_z = \ee ^ {-2z} [(t - \xi _ 0) - 2G - 2 z (t - \xi _ 0)],
  \quad
  \Psi_t = \ee ^ {-2z} [(t - \xi _ 0) + z],
\]
a direct expansion followed by Lemma~\ref{lem:moments} gives
\[
  J(c)
  =
  \frac {1}{32}
    [
      8(1 + c ^ 2)M _ 4
      + (-4c ^ 2 + 8c - 4)C _ 1
      + (2c ^ 2 - 6 c + 6)\Theta _ 0 + 1
    ].
\]
In particular,
\[
  J(0) = \frac {-4C _ 1 + 6\Theta _ 0 + 8M_4 + 1}{32},\quad
  J(1) = \frac {2\Theta _ 0 + 16M _ 4 + 1}{32}.
\]
Substituting the expression for $M_4$ from~\eqref{eq:moment-1234} and using the
bounds in~\eqref{eq:coarse-bounds}, we obtain
\[
  C _ 1 - J(0)
  =
  \frac{(36 + 9\xi _ 0)C _ 1 - 3\Theta _ 0 ^ 2 - 6\Theta _ 0 - 4}{32}
  >
  \frac {803}{12800}
  > 0.
\]
Similarly,
\[
  C _ 1 - J(1)
  = \frac {(32 + 18\xi _ 0)C _ 1 - 6\Theta _ 0 ^ 2 - 2\Theta _ 0 - 7}{32}
  > \frac {203}{6400}
  > 0.
\]
It follows from~\eqref{eq:J-definition} that $J$ is a convex quadratic polynomial
in $c$. Hence, for every $0\leq c\leq1$,
\[
  J(c)
  \leq
  (1-c)J(0)+cJ(1)
  < C_1.
\]
Consequently, for every $k \geq 1$,
\begin{equation}\label{eq:E-final}
  E_k[\chi]
  =
  \frac {1}{k ^ 2} J(1/k)
  <
  \frac{C _ 1}{k ^ 2}
  \leq
  \frac{C _ 1}{k}.
\end{equation}
Combining Proposition~\ref{prop:defect} with \eqref{eq:E-final}, we find that the
limiting energy defect is strictly negative. Hence, for all sufficiently small
$\epsilon>0$, the reflected full-sector trial function satisfies
\[
  q_{\alpha}[u_\epsilon]
  <
  \Theta _ 0 \norm {u_\epsilon} ^ 2.
\]
The Rayleigh principle therefore gives
\[
  \lambda(\alpha) < \Theta _ 0,
  \quad
  \frac{\pi}{2}\leq\alpha<\pi.
\]
Together with the previously known result for $0<\alpha\leq\pi/2$, this completes
the proof of Theorem~\ref{thm:main}.

\begin{remark}\label{rem:quantitative}
Since $\norm {v _ \epsilon} ^ 2 \sim 1/\epsilon$, the finite limit in
Proposition~\ref{prop:defect} controls the Rayleigh quotient only up to an
$O(\epsilon)$ error and therefore does not yield a quantitative estimate on the
gap $\lambda(\alpha) - \Theta _ 0$. Obtaining such a bound would likely require
adjusting the exponential prefactor in $v _ {\epsilon}$ by a function better
adapted to the energy optimization.
\end{remark}

\section {Applications}

We discuss here direct applications of Theorem~\ref{thm:main}.

\subsection{Geometric eigenvalue bound.}

Colbois, L\'ena, Provenzano and Savo established geometric bounds for the
eigenvalues of the magnetic Laplacian in general domains~\cite{MR4659291}. One of
their results concerns polygonal rep-tiles. Recall that a polygonal domain
$\Omega$ is called a rep-tile if it can be tiled by $m\ge 2$ scaled copies
$\omega_i$ of $\Omega$, each isometric to $m^{-1/2}\Omega$
(see~\cite[Definition~3.5]{MR4659291}). If $\Omega$ is a polygonal rep-tile with
interior angles $\alpha_1,\dots,\alpha_N$, then by~\cite[Theorem~2.2]{MR4659291},
\begin{equation}\label{eq:gb}
\lambda_1(\Omega,b \A_0) \leq b\Theta_0,
\end{equation}
where $\lambda_1(\Omega,b\A_0)$ denotes the lowest eigenvalue of the
magnetic Neumann Laplacian on $\Omega$ with constant magnetic field strength
$b>0$. This bound follows from the stronger result \cite[Theorem~3.7]{MR4659291}:
\[
\lambda_1(\Omega,b \mathbf A_0) \le b\min_{1\leq i\leq N}\lambda(\alpha_i).
\]

In the case where the angles $\alpha_1,\dots,\alpha_N\in(0,\pi)$, Theorem~\ref{thm:main} sharpens the bound in \eqref{eq:gb}:
\begin{equation}\label{eq:gb*}
\lambda_1(\Omega,b \A_0) < b\Theta_0.
\end{equation}

\subsection{Onset of superconductivity.}

Theorem~\ref{thm:main} has immediate consequences for the effective models in the
Ginzburg--Landau theory of type-II superconductors. These effective models are
defined in the sector $\Omega_\alpha$ and their validity requires the inequality
$\lambda(\alpha)<\Theta _ 0$. With Theorem~\ref{thm:main} in hand, the
corresponding results hold for all opening angles $\alpha\in(0,\pi)$. Thus,
\begin{enumerate}
  \item Theorem 1.1 and its corollary in~\cite{MR4890869} hold for all opening
  angles in $(0,\pi)$.
  \item For a bounded domain with corners, the third critical field is strictly
  larger than that for smooth domains, as long as the opening angles of the
  corners lie in $(0,\pi)$. Moreover, Theorem~1.2 in~\cite{MR3832147}, describing
  the density of superconductivity near corners, holds without restriction on the
  opening angles.
\end{enumerate}

\subsection{Possible extensions and related problems.}

The proof of Theorem~\ref{thm:main} can likely be adapted to other settings where
corner-like effects appear in magnetic Schr\"odinger operators. A notable example
is the magnetic-step problem, where a discontinuity in the applied magnetic field
produces effects analogous to those caused by corners \cite{MR4201025,
MR4778604}. Another natural direction is the study of three-dimensional wedges,
where analogous spectral questions arise for the magnetic Neumann Laplacian in
domains with edges \cite{MR3416835}. The variational technique developed here may
provide a suitable framework for establishing the analogue of the strict
inequality $\lambda(\alpha) < \Theta _ 0$ in such problems.

\bibliographystyle{amsplain}
\bibliography{KS-sector}

\end{document}